\documentclass{amsart}

\usepackage{amsthm,amsmath,graphicx,epstopdf}

\newcommand{\R}{{\mathbb R}}

\theoremstyle{plain}
\numberwithin{equation}{section}
\newtheorem{thm}{Theorem}
\newtheorem{theorem}[thm]{Theorem}
\newtheorem{lemma}[thm]{Lemma}

\newtheorem{prop}[thm]{Proposition}

\newtheorem{corollary}[thm]{Corollary}

\theoremstyle{definition}

\newtheorem{remark}[thm]{Remark}

\title[Braiding link cobordisms and non-ribbon surfaces]{Braiding link cobordisms and non-ribbon surfaces}

\author{Mark C. Hughes}
\email{hughes@mathematics.byu.edu}

\begin{document}

\maketitle

\begin{abstract} 
We define the notion of a braided link cobordism in $S^3 \times [0,1]$, which generalizes Viro's closed surface braids in $\R^4$.  We prove that any properly embedded oriented surface $W \subset S^3 \times [0,1]$ is isotopic to a surface in this special position, and that the isotopy can be taken rel boundary when $\partial W$ already consists of closed braids.  These surfaces are closely related to another notion of surface braiding in $D^2 \times D^2$, called braided surfaces with caps, which are a generalization of Rudolph's braided surfaces.  We mention several applications of braided surfaces with caps, including using them to apply algebraic techniques from braid groups to studying surfaces in 4-space, as well as constructing singular fibrations on smooth 4-manifolds from a given handle decomposition.  
\end{abstract}

\section{Introduction}
\label{sec:Introduction}

\sloppypar{Two of the most useful and foundational results in knot theory and low-dimensional topology are the classical theorems of Alexander and Markov.  These theorems allow us to study knots entirely within the realm of braids and braid closures, where we can exploit either the algebraic structure of the braid group, the special position of a closed braid in $S^3$, or the fact that braids with isotopic closures can be related by special braid moves.  These results have been used in numerous applications, examples of which include the construction and categorification of quantum link invariants \cite{HOMFLY,Jones,KRII}, the construction of open book decompositions on 3-manifolds \cite{Alexanderopenbook}, and studying the slice and ribbon genera of knots \cite{Rudolph1983,Rudolph1993}.}

The notion of a closed braid as a specially positioned 1-dimensional submanifold of 3-dimensional space has been generalized by different authors to certain classes of surfaces in 4-space.  One such generalization is due to Rudolph \cite{Rudolph1983}, who considered surfaces $S \subset D^2 \times D^2$ on which the projection to the second factor $\mathrm{pr}_2: D^2 \times D^2 \rightarrow D^2$ restrict as branched coverings.  This generalizes the classical notion of a (geometric) braid as a 1-dimensional submanifold of $D^2 \times [0,1]$, on which the projection $\mathrm{pr}_{[0,1]}:D^2 \times [0,1] \rightarrow [0,1]$ restricts as an ordinary covering.  These surfaces are called \emph{braided surfaces}, and are closely related to a similar notion due to Viro \cite{virolecture}.  Any braided surface is necessarily ribbon, and Rudolph showed that every orientable ribbon surface with boundary properly embedded in $D^2 \times D^2$ is isotopic to a braided surface. 

Like their lower-dimensional counterparts, braided ribbon surfaces have found use in various applications, including finding obstructions to sliceness in knot theory \cite{Rudolph1993}, the study of Stein fillings of contact 3-manifolds, and the construction of Lefschetz fibrations on 4-dimensional 2-handlebodies (i.e., 4-manifolds admitting handle decompositions with no 3 or 4-handles).  Indeed, using the fact that any oriented 4-dimensional 2-handlebody $X$ admits a covering over $D^2 \times D^2$ branched along an orientable ribbon surface, Loi and Piergallini \cite{LoiandPiergallini} were able to construct Lefschetz fibrations on $X$, and used them to give a topological characterization of Stein surfaces with boundary.

As Rudolph's braided surfaces do not include non-ribbon surfaces, the above techniques were not sufficient for studying smooth 4-manifolds with 3 or 4-handles.  Indeed, the branched coverings of such manifolds over $D^2 \times D^2$ do not have ribbon branch loci. Expanding these applications thus requires a more general notion of braided surface.

In this paper we generalize these notions further, by defining \emph{braided link cobordisms} (or simply \emph{braided cobordisms}).  These are surfaces $W \subset S^3 \times [0,1]$ smoothly and properly embedded, on which the projection $\mathrm{pr}_2:S^3 \times [0,1] \rightarrow [0,1]$ restricts as a Morse function, with each regular level set $W \cap (S^3 \times \{t\})$ a closed braid in $S^3 \times \{t\}$.  Braided cobordisms generalize Viro's closed 2-braids to oriented surfaces with boundary.  We prove the following:

\begin{theorem}
\label{thm:AlexanderMainTheorem}
Let $W \subset S^3 \times [0,1]$ be an oriented surface smoothly and properly embedded.  Then $W$ is isotopic to a braided cobordism.  If the boundary links of $\partial W$ are already closed braids, then this isotopy can be chosen rel $\partial W$. 
\end{theorem}

Theorem \ref{thm:AlexanderMainTheorem} can be thought of as a cobordism analogue to the classical Alexander's theorem, and will be proven in Section~\ref{sec:BraidingLinkCobordisms}.  Our construction will be similar to Kamada's construction of the normal braid form of a surface link \cite{kamada}, which implies our result in the case that $W$ is a closed surface.  The bulk of the additional work here will be in carrying out the construction in a way that allows us to keep $\partial W$ fixed during the required ambient isotopies.  This boundary-fixing requirement is considered with an eye toward applications (see either \cite{jacobsson} for a construction using Khovanov homology which is not invariant under general isotopies of $W$, or below for other applications).

We also define a related class of surfaces in $D^2 \times D^2$, called \emph{braided surfaces with caps}, which generalize Rudolph's braided surfaces (see Section~\ref{sec:BraidedSurfacesWithCaps}), and which are closely related to braided cobordisms.  Theorem~\ref{thm:AlexanderMainTheorem} then gives us the following:

\begin{corollary}
\label{cor:braidedsurfaceswithcaps}
Let $S$ be a smooth oriented properly embedded surface in $D^2 \times D^2$.  Then $S$ is isotopic to a braided surface with caps.  If $\partial S$ is already a closed braid, then the isotopy can be chosen rel $\partial S$.
\end{corollary}

These generalized surface braiding results make it possible to extend applications which rely on Rudolph's braiding algorithm.  Here we outline one such application, which involves extending Loi and Piergallini's techniques to construct broken Lefschetz fibrations on oriented smooth 4-manifolds.  Let $X$ be a smooth, oriented, compact 4-manifold, and $\Sigma$ a compact oriented surface.  Then a surjective map $f:X \rightarrow \Sigma$ is called a \emph{Lefschetz fibration} if around every critical point the map $f$ can be modeled in orientation-preserving complex coordinates locally as $f(u,v)=u^2+v^2$.  It is called a \emph{broken Lefschetz fibration}, if along with these isolated critical points, it also contains embedded circles of critical points near which $f$ is locally modeled by  $f(\theta,x,y,z) = (\theta, x^2+y^2-z^2)$.  

Lefschetz fibrations are closely related to symplectic structures on $X$ \cite{Donaldson,GompfandStipsicz}, and allow us to express the 4-manifold $X$ combinatorially in terms of the monodromy of a regular fiber (cf.\ \cite{GompfandStipsicz}).  Broken Lefschetz fibrations exist more generally, but share a similar relation to near-symplectic structures \cite{AurouxDonaldsonKatzarkov}, and can be used to define invariants of smooth 4-manifolds and finitely presented groups \cite{Baykur2012}.  They were introduced by Auroux, Donaldson, and Katzarkov in \cite{AurouxDonaldsonKatzarkov}, where they constructed a broken Lefschetz fibration on $S^4$.  Later, it was shown independently by Akbulut and Karakurt \cite{AkbulutandKarakurt}, Baykur \cite{Baykur}, and Lekili \cite{Lekili}  that any oriented smooth 4-manifold admits a broken Lefschetz fibration over $S^2$.  Although their approaches differ, none of them build the desired fibration directly from a given handle decomposition of $X$, instead relying on the modification of critical points of generic maps, or deep classification results from contact topology.

Corollary~\ref{cor:braidedsurfaceswithcaps} allows us to extend Loi and Piergallini's techniques to construct broken Lefschetz fibrations from handle decompositions on a wide class of 4-manifolds.  Indeed, given a handle decomposition of $X$ with $\partial X \neq \emptyset$, we can construct a branched covering $h:X \rightarrow D^2 \times D^2$ one handle at a time, so that the branch locus is a surface with only cusp and node singularities.  In many cases this branch locus can be made to be orientable, and hence by Corollary~\ref{cor:braidedsurfaceswithcaps} can be isotoped to a braided surface with caps in $D^2 \times D^2$.  The desired fibration on $X$ is then obtained as the composition $\mathrm{pr}_2\circ h:X \rightarrow D^2$.  This construction yields fibrations directly from the handle decomposition of $X$, and can be combined with techniques in \cite{GayandKirby} to give broken Lefschetz fibrations on closed 4-manifolds. 

Another avenue of application lies in using algebraic information from a braid to answer geometric questions about its closure.  Indeed, Rudolph used braided ribbon surfaces to study quasipositive links \cite{Rudolph1985,Rudolph1993,Rudolph2005} (links which bound braided ribbon surfaces with only positive branch points), as well as to find bounds on the ribbon genus of a link in terms of algebraic information from the braid group \cite{Rudolph1983}.  Using braided (non-ribbon) surfaces with caps, this latter approach can be extended further to look for bounds on the genus of an arbitrary surface bounded by a link, in terms of algebraic information from its boundary.  Furthermore, there are a number of link invariants whose definitions require they be computed on closed braid diagrams (e.g.\ \cite{KRII}).  By examining links that are joined by a given braided cobordism $W$, one could attempt to extend these invariants across $W$, and uncover interesting relationships between the invariants along $\partial W$ and the surface $W$.  The author intends to pursue these questions further in upcoming work.

The remainder of this paper will be organized as follows.  In Section~\ref{sec:SurfaceBraidings} we define various notions of surface braidings in $D^2 \times D^2$ and $S^3\times [0,1]$, as well as outline the relationship between them.  In Section~\ref{sec:BraidingLinkCobordisms} we present diagrammatic methods for studying 1-dimensional braids and surfaces in 4-space, and use them to prove Theorem~\ref{thm:AlexanderMainTheorem} and Corollary~\ref{cor:braidedsurfaceswithcaps}.  

\section{Braided surfaces in 4-space}
\label{sec:SurfaceBraidings}

\subsection{Links as braid closures}
Let $D^2 \subset \mathbb{C}$ be the closed unit disk, $S^1 = \partial D^2$, and $S^3 = \{(z,w) : |z|^2+|w|^2 = 1\}\subset \mathbb{C}^2$ the unit 3-sphere.  We set $T_1 = S^3 \cap \{|z|\leq \frac{1}{\sqrt{2}}\}$ and $T_2 = S^3 \cap \{|w|\leq \frac{1}{\sqrt{2}}\}$, which are both tori, and let $U = S^3 \cap \{w=0\}$ (i.e., the core of $T_2$).  We say that an oriented link $L$ in $S^3$ is a \emph{closed braid} if $L \subset S^3 \backslash U$, and $\mathrm{arg}(w)$ is strictly increasing as we traverse the components of $L$ in the positively oriented direction.  We call $U$ the \emph{axis} of the closed braid.  

Alexander's theorem then says that any oriented link in $S^3$ is isotopic to a closed braid.  Markov's theorem says that any two closed braids which are isotopic as links can be joined by a sequence of isotopies through closed braids, as well as stabilization and destabilizations moves which increase and decrease the braid index respectively.

\subsection{Movie presentations of braided cobordisms}
\label{sec:MoviePresentations}
Recall from Section~\ref{sec:Introduction} that a braided cobordism is a surface $W \subset S^3 \times [0,1]$ smoothly and properly embedded, on which the projection $\mathrm{pr}_2:S^3 \times [0,1] \rightarrow [0,1]$ restricts as a Morse function, with each regular level set $W_t=W \cap (S^3 \times \{t\})$ a closed braid in $S^3 \times \{t\}$.  We will assume in what follows that $\mathrm{pr}_2|_W$ is injective on its set of critical points.  Each regular $W_t$ with $t<1$ is oriented as the boundary of $W\cap(S^3 \times\left[t,1\right])$.  

We now establish a diagrammatic method for describing braided cobordisms.  Choose a point $p \in U \subset S^3$ with $\{p\} \times [0,1]$ disjoint from $W$, and identify the complement of $p$ in $(S^3,U)$ with $(\mathbb{R}^3,z\mathrm{-axis})$.  Choose the identification so that $\mathrm{arg}(w)$ corresponds with the angular cylindrical coordinate on $\mathbb{R}^3$.  Here we let $(x,y,z)$ denote the usual coordinates on $\mathbb{R}^3$, while $t$ denotes the coordinate on $[0,1]$. 

Let $\pi : \mathbb{R}^3 \rightarrow \mathbb{R}^2$ denote the orthogonal projection to the $xy$-plane.  After perturbing $W$ slightly if necessary, we can assume that $\pi \times \text{id}:\mathbb{R}^3 \times \left[0,1\right] \rightarrow \mathbb{R}^2 \times \left[0,1\right]$ restricts to a family of regular link projections $W_t\rightarrow \R^2 \times [0,1]$ for all but finitely many $t \in \left[0,1\right]$.  After decorating with over and under crossing information, we obtain a continuous family of link diagrams with finitely many singular diagrams.  As each regular $W_t$ is a closed braid, each regular diagram will be the diagram of a closed braid, while passing a singular still will change the diagram by either:
\begin{enumerate}
\item addition or deletion of a single loop around $0 \in \mathbb{R}^2$ disjoint from the rest of the diagram (corresponding to local maximum and minimum points of $W$),
\item addition or deletion of a single crossing between adjacent strands in the braid diagram by a band surgery (corresponding to saddle points of $W$),
\item a single braid-like Reidemeister move of type II or III, where each strand involved in the move is oriented in the positive direction. 
\end{enumerate}

We refer to this family of link diagrams as the \emph{movie presentation} of $W$.  Note that because we are not assuming $W$ is in general position with respect to the $z$ and $t$-projections, our definition of movie presentation differs slightly from that used by other authors (see e.g. \cite{surfacesinfourspace}).  Note that during the proof of Theorem~\ref{thm:AlexanderMainTheorem} we will also consider movie presentations using projections other than the orthogonal projection $\pi:\R^3 \rightarrow \R^2$ to the $xy$-plane.

The surface $W$ can then be described by taking a finite number of the nonsingular stills, where each one differs from the previous still by a single modification as described above, or by a planar isotopy preserving the closed braid structure.  Some caution is needed in using such descriptions, as different choices of planar isotopies linking two adjacent diagrams can result in non-isotopic embeddings (see e.g., \cite{jacobsson}).   See Figure \ref{fig:braidcobord} for a genus 1 example of a braided movie presentation between the trefoil and the empty knot (the stills are read as lines of text, from left to right).

\begin{figure}
 \centering
 \includegraphics[width=\textwidth]{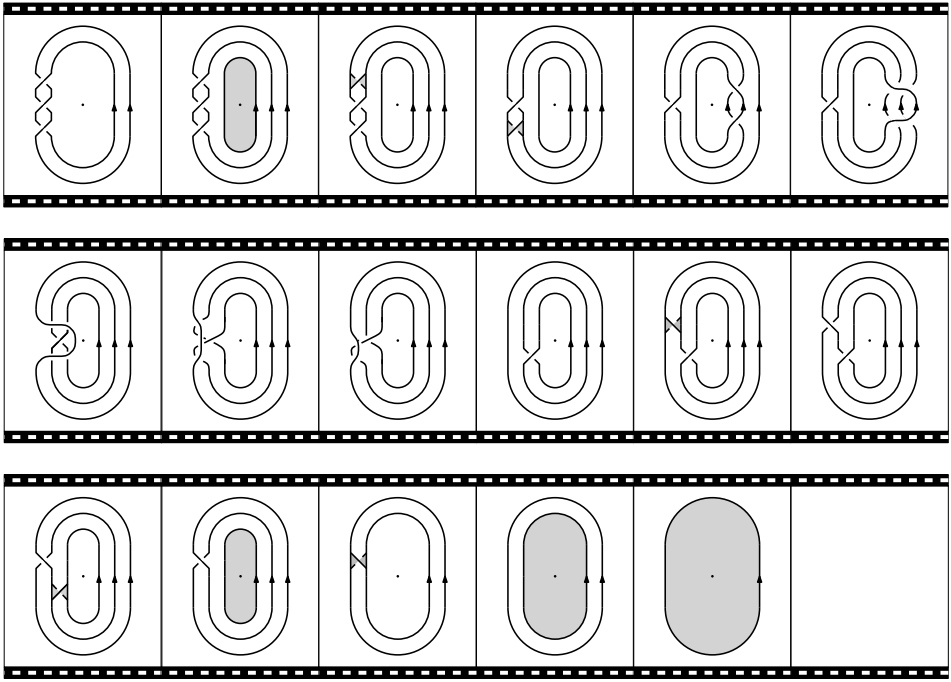}
  \caption{Braided movie presentation}
  \label{fig:braidcobord}
\end{figure}

\subsection{Braided surfaces in $D^2 \times D^2$} 
Rudolph defined a \emph{braided surface} \cite{Rudolph1983} to be a smooth properly embedded oriented surface $S \subset D^2 \times D^2$ on which the projection to the second factor $\mathrm{pr}_2:D^2 \times D^2 \rightarrow D^2$ restricts as a simple branched covering
.  Examples of these braided surfaces can be obtained by taking intersections of non-singular complex plane curves with 4-balls in $\mathbb{C}^2$, and they can be used to study the links that arise as their boundaries in $S^3 = \partial D^4$ (see e.g. \cite{Rudolph1985,Rudolph1993,Rudolph2005}).  

Let $S$ be a braided surface.  In a neighborhood of any branch point $p$ of the covering $\mathrm{pr}_2|_S$, there are local complex coordinates $u$ and $v$ on $D^2$ such that $S$ is given by the equation $u^2=v$, in the coordinates $(u,v)$ on $D^2 \times D^2$.  

The boundary of $D^2 \times D^2$ decomposes as $\partial (D^2 \times D^2) = (D^2 \times S^1) \cup (S^1 \times D^2)$ in the obvious way, and we set $\partial_1 = D^2 \times S^1$ and $\partial_2 = S^1 \times D^2$.  We then define closed braids in $\partial (D^2 \times D^2)$ as links in $\partial_1$ on which the projection $\mathrm{pr}_2 : \partial_1 \rightarrow S^1$ restricts to a covering map.  Notice then that the boundary of a braided surface is a closed braid in $\partial (D^2 \times D^2)$.

One feature of Rudolph's braided surfaces are that they are all necessarily \emph{ribbon}.  A properly embedded surface $S$ in $D^4 = \{(z,w) : |z|^2+|w|^2 \leq 1\}$ is said to be \emph{ribbon embedded} if the function $|z|^2+|w|^2$ restricts to $S$ as a Morse function with no local maximal points on $\text{int }S$.  A properly embedded surface in $D^4$ is said to be \emph{ribbon} if it is isotopic to a surface which is ribbon embedded.  By fixing an identification of $D^2 \times D^2$ with $D^4$, we can similarly consider ribbon surfaces in $D^2 \times D^2$ (the definition of ribbon embeddings in $D^2 \times D^2$ will depend on our choice of identification, though the resulting class of ribbon surfaces will not).  

Rudolph proved that any orientable ribbon surface in $D^2 \times D^2$ is isotopic to a braided surface, though in general this isotopy cannot be chosen to fix $\partial S$, even if $\partial S$ is already a closed braid.

Viro defined a similar notion which he called a \emph{2-braid}, by additionally requiring that $\partial S \subset \partial_1 = D^2 \times S^1$ be a trivial closed braid (i.e., $\partial S = P \times S^1$ for some finite subset $P \subset D^2$).  $2$-braids come equipped with a closure operation yielding closed surfaces in $S^4$, and Viro \cite{virolecture} proved a 4-dimensional Alexander theorem by showing that every closed oriented surface in $S^4$ is isotopic to the closure of a $2$-braid.  These 2-braids were also studied extensively by Kamada \cite{kamada1993,alexandermarkovdim4,kamada1996,kamada1999,kamada}, who proved a 4-dimensional Markov theorem relating any two 2-braids with isotopic closures.

\subsection{Braided surfaces with caps}
\label{sec:BraidedSurfacesWithCaps}
The embedded surfaces in $D^2 \times D^2$ we consider in this paper will not in general be ribbon, and hence cannot be braided via Rudolph's algorithm.  We thus consider a less restrictive notion of braiding, which we define now.

Let $\phi:F \rightarrow \Sigma$ be a smooth map of oriented surfaces.  Then a \emph{cap of $F$ with respect to} $\phi$ is an  embedded disk $D \subset F$, so that 
\begin{enumerate}
\item $\phi$ restricts to embeddings on $\text{int }D$ and on $\partial D$,
\item $F$ and $\Sigma$ both admit coordinate charts of the form $S^1 \times [-1,1]$ around $\partial D = S^1 \times \{0\}$ and $\phi(\partial D) = S^1 \times \{0\}$, on which $\phi$ is given by $(\theta,t) \mapsto (\theta,t^2)$,
\item in the above coordinate chart around $\phi(\partial D)$, the curve $S^1 \times \{1\}$ lies in $\phi(\text{int }D)$.  
\end{enumerate}

Now let $S \subset D^2 \times D^2$, and let $\mathrm{pr}_S$ denote the restriction of $\mathrm{pr}_2$ to $S$.  We say that $S$ is a \emph{braided surface with caps} if the critical points of $\mathrm{pr}_S$ all correspond either to isolated simple branch points or boundaries of caps of $S$ with respect to $\mathrm{pr}_S$.  Moreover, we will often assume that the critical values in $D^2$ form a set of embedded concentric circles (corresponding to the boundaries of caps), with isolated critical values  lying inside the innermost circle.  See Figure \ref{fig:braidsurfacewithcaps} for a cross sectional diagram of a braided surface with a single cap.  

\begin{figure}
 \centering
 \includegraphics[width=7cm]{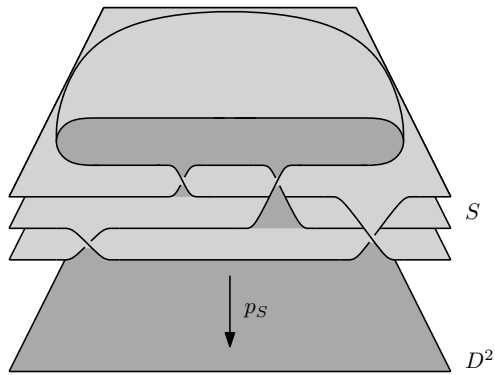}
  \caption{Cross section of a braided surface with caps}
  \label{fig:braidsurfacewithcaps}
\end{figure}

\subsection{Braided surfaces with caps from braided cobordisms}
\label{sec:BraidedCobordismsFromBraidedSurfaces}

Braided cobordisms  are closely related to braided surfaces with caps, a fact which we illuminate here.  We begin by defining a smooth map $\rho : S^3 \rightarrow D^2$ as follows.  Let $\lambda : \left[0,1\right] \rightarrow \left[0,1\right]$ be a smooth function with $\lambda(t) = t$ on $\left[0,\frac{1}{4}\right]$, $\lambda \equiv \frac{1}{\sqrt{2}}$ on $[\frac{1}{\sqrt{2}},1]$, and so that $\frac{d\lambda}{dt} > 0$ on $[0,\frac{1}{\sqrt{2}})$.  Then we define $\rho : S^3 \rightarrow D^2$ as 
\[
\rho(z,w) = \frac{\sqrt{2}w \lambda (|w|)}{|w|}
\]
for $w \neq 0$, and $\rho(z,0) = 0$.  Clearly $\rho$ is smooth, with $T_1=\rho^{-1}(\partial D^2)$ and $T_2=\overline{\rho^{-1}(\text{int }D^2)}$.  Furthermore, using $\rho$ we can fix a fibering of $T_1$ over $S^1$ with fiber $D^2$, and a fibering of $T_2$ over $D^2$ with fiber $S^1$.  A link $L \subset T_1$ is a closed braid if and only if $\rho|_L : L \rightarrow S^1$ is a covering map.  We call the degree of the covering map $\rho|_L$ the \emph{index} of the closed braid $L$. 


We now identify $\partial (D^2 \times D^2)$ with $S^3$ by a smooth homeomorphism $\kappa: \partial (D^2 \times D^2) \rightarrow S^3$, which smooths the corners of $\partial (D^2 \times D^2)$, and identifies $\partial_1$ with $T_1$ and $\partial_2$ with $T_2$.  Furthermore, we assume that $\kappa$ is a diffeomorphism away from the corners of $\partial (D^2 \times D^2)$, and maps the fibers of $\mathrm{pr}_2$ diffeomorphically onto the fibers of $\rho$.

For $0 \leq t \leq 1$, we can multiply $\partial (D^2 \times D^2) \subset \mathbb{C}^2$ by a factor of $\frac{1}{2}(t+1)$ and use $\kappa$ to identify the resulting set with $S^3 \times \{t\}$.  We thus obtain an identification of $S^3 \times [0,1]$ with a collar neighborhood $\nu$ of $\partial (D^2 \times D^2)$ in $D^2 \times D^2$, which we denote by $\kappa': \nu \rightarrow S^3 \times [0,1]$. 

As any properly embedded surface $S$ in $D^2 \times D^2$ can easily be arranged to lie in the collar neighborhood $\nu$, we see that after smoothing corners any such surface gives rise to a smooth properly embedded surface in $S^3 \times [0,1]$ whose boundary lies in $S^3 \times \{1\}$, and vice versa.

\begin{lemma}
\label{lem:braidedcobordismtobraidedsurfacewithcaps}
Suppose that $W\subset S^3 \times [0,1]$ is a braided cobordism, with $W \cap (S^3 \times \{0\}) = \emptyset$.  Then $(\kappa')^{-1} (W)$ will be a braided surface with caps in $D^2 \times D^2$ (after a small isotopy smoothing corners around the boundaries of the caps).
\end{lemma}

\begin{proof}
Let $S = (\kappa')^{-1}(W)$, and let $\mathrm{pr}_S$ denote the restriction of $\mathrm{pr}_2$ to $S$.  Each local maximum or minimum point of $W \subset S^3 \times [0,1]$  with respect to the height function will lie in $T_2 \times [0,1]$, and we can arrange that each saddle point of $W$ lies in $T_1\times [0,1]$.  Furthermore, by flattening a neighborhood of each local maximum and minimum point, we can isotope $W$ so that it intersects $T_2 \times [0,1] = S^1 \times D^2 \times [0,1]$ in a collection of disks of the form $\{p\}\times D^2 \times \{t\}$.  The image of any such disk under $(\kappa')^{-1}$ will be a disk in $\frac{1}{2}(t+1) \cdot \partial_2$, and the restriction of $\mathrm{pr}_S$ to its interior will be free of critical points.

Now $W'_t=W \cap (T_1 \times \{t\})$ will be a (possibly singular) closed braid in $T_1 \times \{t\}$ for each $0\leq t \leq 1$.  Each singular braid $W'_t$ will consist of a closed braid with a pair of strands intersecting at a point, with distinct tangent lines.  These self-intersections corresponds to saddle points of the surface $W$.  Each $(\kappa')^{-1}(W'_t)$ will thus also be a possibly singular closed braid in $\frac{1}{2}(t+1)\cdot\partial_1$, where each singular point gives rise to a simple branch point of the projection $\mathrm{pr}_S$.  The non-singular points of these closed braids all correspond to regular points of $\mathrm{pr}_S$.

Finally, it remains to consider what happens along the boundaries of the disks in $W \cap (T_2 \times [0,1])$.  For any disk $D$ corresponding to a local minimum of $W$, the boundary of $(\kappa')^{-1}(D)$ can be smoothed in such a way that the resulting points are all regular points of the map $\mathrm{pr}_S$.  If $D$ instead corresponds to a local maximum, then the boundary of $(\kappa')^{-1}(D)$ is instead smoothed in such a way that $(\kappa')^{-1}(D)$ becomes a cap of $S$ with respect to $\mathrm{pr}_S$.  Since all critical points of $\mathrm{pr}_S$ are either isolated simple branch points, or lie along the boundary or a cap, $S \subset D^2 \times D^2$ is a braided surface with caps.
\end{proof}

\section{Braiding link cobordisms}
\label{sec:BraidingLinkCobordisms}

We begin with the proof of Theorem \ref{thm:AlexanderMainTheorem}.  For the duration of the proof, it will be convenient to think of our cobordisms as lying in $\R^3 \times [0,1]$ so that we can use the diagrammatic approach described in Section~\ref{sec:MoviePresentations}.  Suppose that $W \subset \mathbb{R}^3 \times \left[ 0, 1 \right]$ is a properly embedded oriented link cobordism between closed braids $B_0 \subset \mathbb{R}^3 \times \{ 0 \}$ and $B_1 \subset \mathbb{R}^3 \times \{ 1 \}$.  Assume furthermore that the restriction of the projection $\mathrm{pr}_2 : \R^3 \times [0,1] \rightarrow [0,1]$ to $W$ is a Morse function.  For any such surface $W \subset \mathbb{R}^3 \times [0,1]$ and any $[a,b]\subset [0,1]$, let $W_{[a,b]} = W \cap (\mathbb{R}^3 \times [a,b])$, and $W_t = W \cap (\R^3 \times \{t\})$.

\subsection{Braiding around critical points}
\label{sec:BraidingAroundCriticalPoints}

We begin by proving that $W$ can be ``braided" in a neighborhood of the critical points of $\mathrm{pr}_2|_W$.  This will reduce the problem of proving Theorem~\ref{thm:AlexanderMainTheorem} to proving it for cobordisms $W$ without critical points.

\begin{lemma}
There is an isotopy of $W$ rel $\partial W$, taking $W$ to a surface $W'$ such that $W'_{[a,b]}$ is a braided cobordism for $[a,b] \in \{[0,\frac{1}{6}], [\frac{1}{3},\frac{2}{3}],[\frac{5}{6},1]\}$, and is free of critical points for $[a,b] \in \{[\frac{1}{6},\frac{1}{3}],[\frac{2}{3},\frac{5}{6}]\}$.
\end{lemma}
 
\begin{proof}
As both $B_0$ and $B_1$ are closed braids, $W_t$ will also be a closed braid for $t$ close to $0$ and $1$, and so we can assume that $W_t$ is a closed braid for all $t \in \left[0,\frac{1}{6}\right] \cup \left[\frac{5}{6},1\right]$.  Push all minimal points into $\mathbb{R}^3 \times \left[0,\frac{1}{6}\right]$, all maximal points into $\mathbb{R}^3 \times \left[\frac{5}{6},1\right]$, and all saddle points into $\mathbb{R}^3 \times \left\{\frac{1}{2}\right\}$ (see \cite{kamada} for details).  The maximal and minimal points can easily be positioned in such a way that $W'_{[0,\frac{1}{6}]}$ and $W'_{[\frac{5}{6},1]}$ remain braided.

Now passing each saddle point changes the level set $W_t$ by surgery along a 2-dimensional 1-handle.  After a small perturbation in a neighborhood of each saddle point, we can assume that these 1-handles all lie in $\mathbb{R}^3 \times \left\{\frac{1}{2}\right\}$.  By adding a half-twist in each band, we can arrange that each segment of $W_{\frac{1}{2}+\varepsilon}$ and $W_{\frac{1}{2}-\varepsilon}$ involved in the surgeries are oriented in the positive direction (see Figure \ref{fig:saddlesarrange}, where $W_{\frac{1}{2}}$ is shown).  Keeping these bands in place, the remaining strands of $W_{\frac{1}{2}}$ can be braided using the standard proof of the classical Alexander's theorem.  Thus we can arrange $W_{\frac{1}{2}}$ so that it is a closed braid both before and after the surgeries, and can extend the closed braid structure to the rest of $W'_{[\frac{1}{3},\frac{2}{3}]}$.  
\end{proof} 
\begin{figure}
 \centering
 \includegraphics[width=0.35\textwidth]{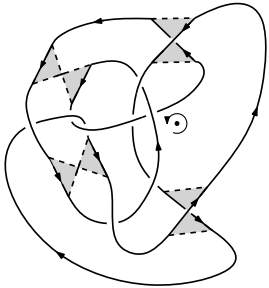}
 \caption{Arranging 1-handles}
  \label{fig:saddlesarrange}
\end{figure}

The above argument is due to Kamada \cite{kamada}.

\subsection{Braiding critical point free cobordisms}
\label{sec:threadingconstruction}

Any cobordism $W$ which is free of critical points is topologically just a union of cylinders, and is isotopic to a product cobordism.  In general, however, the isotopy taking $W$ to a product cobordism cannot be chosen to fix the boundary.  Consider, for example, the movie presentation of the critical point free cobordism $W$ depicted in Figure~\ref{fig:twistspun} (where the middle still is meant to imply that the bottom strand is given a non-zero number of full twists as we look at the level sets moving down).  Here, $W$ is isotopic to a product cobordism, but there is no such isotopy fixing $\partial W$.

\begin{figure}
 \centering
 \includegraphics[width=\textwidth]{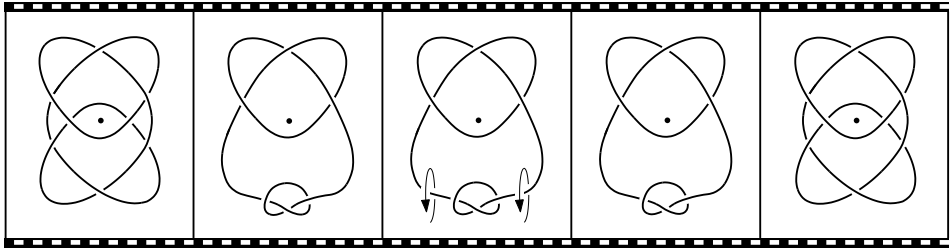}
 \caption{Critical point free cobordism not isotopic rel boundary to product cobordism}
  \label{fig:twistspun}
\end{figure}

The movie presentations of a critical point free cobordism is described entirely by its starting diagram and a sequences of Reidemeister moves and planar isotopies.  We will complete the proof of Theorem \ref{thm:AlexanderMainTheorem} in two stages, first by proving it for critical point free cobordisms whose movie presentation is described entirely by a planar isotopy (i.e., no Reidemeister moves take place between nearby stills) before proving it for the general case.  Before doing this however, we must first recall a geometric set of Markov moves for classical links used by Morton in \cite{morton}, as well as his threading construction which gives a diagrammatic approach to studying isotopies of closed braids.   The proof of Theorem~\ref{thm:AlexanderMainTheorem} relies on enhancements of the arguments used in his proof of Markov's theorem.

\subsection{Geometric Markov moves for closed braids in $\mathbb{R}^3$}
Morton's geometric formulation of Markov's theorem states that two closed braids which are isotopic as links can be joined by a sequence of braid isotopies and simple Markov equivalences.  A \emph{braid isotopy} between two closed braids $L_0$ and $L_1$ in $\mathbb{R}^3$ is an isotopy $\phi_{\alpha}$ of $\mathbb{R}^3$, i.e., a continuous family of maps $\phi_{\alpha}:\mathbb{R}^3 \rightarrow \mathbb{R}^3$ parametrized by $\alpha \in [0,1]$ with $\phi_0 = \mathrm{id}_{\mathbb{R}^3}$, such that $\phi_{\alpha}(L_0)$ is a closed braid for all $\alpha$, and $\phi_1(L_0) = L_1$.

The second move on closed braids is a geometric version of braid stabilization.  Let $B$ and $B'$ be closed braids, and suppose there is an oriented embedded disk $R \subset \mathbb{R}^3$ intersecting the $z$-axis transversely in a single point.  Suppose also that $\partial R = c \cup c'$, where $c=B \cap R$ and $c'=B' \cap R$ are connected and where the boundary orientation of $\partial R$ is winding clockwise along $c$, and counterclockwise along $c'$.  Suppose further that $B\backslash c = B' \backslash c'$.  Then $B$ and $B'$ are said to be \emph{simply Markov equivalent} (see Figure \ref{fig:simplemarkovequivalence} where the disk $R$ is shaded).

The projections of such $B$ and $B'$ to the $xy$-plane differ by a sequence of Reidemeister moves which includes precisely one move of type I creating an extra loop around the origin.

\begin{figure}
 \centering
 \includegraphics[width=0.6\textwidth]{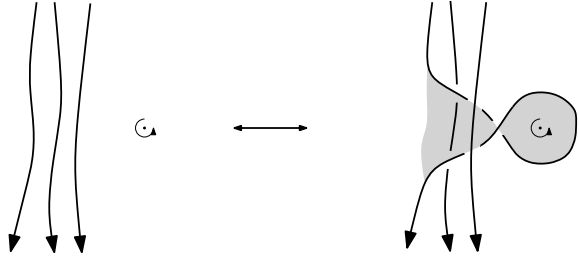}
  \caption{Simple Markov equivalence}
  \label{fig:simplemarkovequivalence}
\end{figure}

\subsection{Threading construction}

Let $P = \{xz\text{-plane}\}$ and let $\pi': \mathbb{R}^3 \rightarrow P$ be the orthogonal projection.  Let $h \subset P$ be the image of the $z$-axis under $\pi'$.  Suppose $D$ is the diagram in $P$ of an oriented link $L$.  Let $C\subset D$ denote the double points (crossings) of $L$ under the projection $\pi'$.    

A \emph{choice of overpasses} for $D$ is a pair of disjoint finite subsets $S,F \subset D\backslash C$, so that each link component contains a points from $S \cup F$, and so that points of $S$ alternate with points of $F$ when traveling along any component.  Furthermore when traveling in the positively oriented direction, each arc of the form $\left[s,f \right]$ contains no undercrossings, and each arc $\left[f,s\right]$ contain no overcrossings.

Now let $P_+ = P \cap \{x>0\}$ and $P_- = P \cap \{x<0\}$ be the right and left-hand regions of $P$ separated by $h$ respectively.  Although $h$ is not a component of the link $L$, we can enhance the diagram $D$ by assigning crossing choices whenever $D$ intersects $h$ transversely.  

Given such an enhanced diagram, $h$ is said to \emph{thread} the diagram $D$ for some choice of overpasses $\left( S,F \right)$, if $h$ intersects $D$ transversely, $S \subset P_-$, $F \subset P_+$, and 
\begin{enumerate}
\item when traveling from $P_-$ to $P_+$, $D$ crosses over $h$, 
\item when traveling from $P_+$ to $P_-$, $D$ crosses under $h$. 
\end{enumerate} 
Threadings of link diagrams allow us to study closed braids on the level of link diagrams.  The following lemma is due to Morton (see \cite{morton}): 

\begin{lemma}
\label{lem:threading}
Suppose $D$ is a diagram that is threaded by $h$ for some choice of overpasses.  Then there is a closed braid $L$ with diagram $D$.   
\end{lemma}

The idea behind the proof of the lemma is summarized in Figure \ref{fig:threading}.  Note that even if the over/under crossing information of $D$ with $h$ has not been specified, there is a unique assignment to each such crossing so that the resulting diagram lifts to a closed braid.  Conversely, it is also easy to show that any closed braid is braid isotopic to one whose diagram is threaded by $h$ for some choice of overpasses.

\begin{figure}
 \centering
 \includegraphics[width=\textwidth]{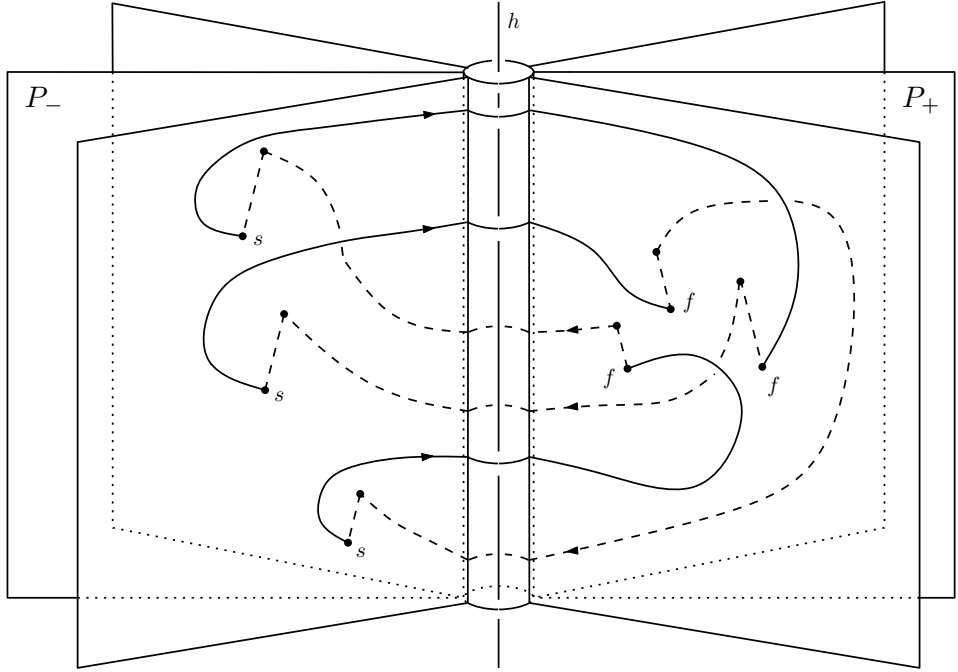}
  \caption{Trefoil as a closed braid given by a threading}
  \label{fig:threading}
\end{figure}

\subsection{Braiding movie presentations without Reidemeister moves}

Now suppose that $W \subset \mathbb{R}^3 \times \left[0,1\right]$ is a critical point free cobordism between two closed braids, and consider the movie presentation of $W$, this time projecting each $W_t \subset \mathbb{R}^3 \times \{t\} = \mathbb{R}^3$ to the plane $P$ via the projection $\pi'$.  We let $D_t$ denote the (possibly singular) diagram of $W_t$ in $P$ for each $t \in \left[0,1\right]$.  As $W$ is free of critical points, nearby diagrams will differ by either a planar isotopy or Reidemeister move.  If the movie presentation of $W$ does not involve any Reidemeister moves, then it can be described completely by specifying the initial diagram $D_0$ and a planar isotopy $\phi_{\alpha}$ of $P$, with $\phi_{\alpha} (D_0) = D_{\alpha}$ for all $\alpha$.  In what follows it will be convenient to specify the movie presentations of such surfaces in this way.      

We prove Theorem \ref{thm:AlexanderMainTheorem} first in the special case when $D_0$ and $D_1$ are threaded, and the movie presentation of $W$ does not involve any Reidemeister moves:

\begin{prop}
\label{prop:planarisotopy}
Suppose $W$ has no critical points, and that its movie presentation does not involve any Reidemeister moves.  Suppose further that $W_0$ and $W_1$ are closed braids with diagrams $D_0$ and $D_1$ threaded by $h$ for some choices of overpasses.  Then $W$ is isotopic relative its boundary to a braided cobordism.  
\end{prop}

In order to prove the above proposition we will need to lift the planar isotopy joining $D_0$ and $D_1$ to a sequence of braid isotopies and simple Markov equivalences in $\mathbb{R}^3$.  For the rest of this section we assume $W$ is as described in the statement of Proposition \ref{prop:planarisotopy}.  The first lemma we will need is the following:   

\begin{lemma}
\label{lem:braidisotopylifting}
Let $\psi_{\alpha}$ be a planar isotopy of $P$ taking $D_0$ to $D_1$ which fixes $h$ setwise.  Suppose further that $\psi_{\alpha} \equiv \psi_0$, and $\psi_{1-\alpha} \equiv \psi_1$ for $\alpha$ in a small neighborhood of 0.  Then there is a braid isotopy $\phi_{\alpha}$ taking $W_0$ to $W_1$, such that $\pi' \circ\phi_{\alpha} (W_0) = \psi_{\alpha}(D_0)$ for all $\alpha \in \left[0,1\right]$.
\end{lemma}

\begin{proof}
For any $p \in W_0$ and $\alpha \in \left[0,1\right]$, the $x$ and $z$-coordinate of $\phi_{\alpha} (p)$ are determined by $\psi_{\alpha}$.  The $y$-coordinate of $\phi_{\alpha} (p)$ can then be chosen uniquely so that the radial coordinate of $\phi_{\alpha}(p)$ remains constant for all $\alpha$.  It thus suffices to note that any two closed braids with the same diagram are also braid isotopic, via a straight line isotopy. 
\end{proof}

Let $(S_0,F_0), (S_1,F_1) \subset P$ denote the overpasses chosen for the threadings of $D_0$ and $D_1$ respectively, and let $\psi_{\alpha}$ denote a planar isotopy of $P$ associated to the movie presentation of $W$, i.e., $\psi_{\alpha}(D_0) = D_{\alpha}$ for all $\alpha \in \left[0,1\right]$.  We can assume that 
\[
S_0 \cap \psi^{-1}_1(S_1) =  F_0 \cap \psi^{-1}_1(F_1) = \emptyset.
\] 

The following lemma will allow us to assume that the choices of overpasses for both $D_0$ and $D_1$ coincide, and that they can be assumed to be fixed by the planar isotopy $\psi_{\alpha}$.  

\begin{lemma}
\label{lem:movingoverpasses}
$W$ is isotopic relative its boundary to a cobordism whose movie presentation is determined by the diagram $D_0$ and a planar isotopy $\varphi_{\alpha}$, where $\varphi_{\alpha}(S_0) = S_0$ and $\varphi_{\alpha}(F_0) = F_0$ for $0 \leq \alpha \leq 1/2$, and where $\varphi_{\alpha}(S_1) = S_1$ and $\varphi_{\alpha}(F_1) = F_1$ for $1/2 \leq \alpha \leq 1$.
\end{lemma}

\begin{proof}
We can assume that for all $q \in S_1\cup F_1$, the sets $\{\psi^{-1}_{\alpha}(q) \text{ } | \text{ } 0 \leq \alpha \leq 1 \}$ are disjoint embedded arcs in $P$ which do not intersect $S_0 \cup F_0$ (see for example Lemma 10.4 of \cite{Burde}).  For each $q \in S_1\cup F_1$ choose a small regular neighborhood $A_q$ of $\{\psi^{-1}_{\alpha}(q) \text{ } | \text{ } 0 \leq \alpha \leq 1 \}$, so that the $A_q$ are pairwise disjoint and also do not intersect $S_0 \cup F_0$.  

Now let $\xi_{\alpha}$ be a planar isotopy of $P$ which restricts to the identity on the complement of $\bigcup A_q$, and such that for all $\alpha \in \left[0,1\right]$ and all $p \in \psi^{-1}_1 (S_1 \cup F_1)$ we have $\xi_{\alpha}(p) = \psi^{-1}_{1-\alpha}\circ \psi_1(p)$.  Let $\Gamma_{\tau,\alpha}$ be the one parameter family of planar isotopies of $P$, with $\tau \in [0,1]$, defined by 
\[
 \Gamma_{\tau,\alpha} =
  \begin{cases}
   \xi_{2\tau \alpha} & \text{if } 0 \leq \alpha \leq 1/2 \\
   \xi_{\tau(2-2\alpha)}  & \text{if } 1/2 \leq \alpha \leq 1.
  \end{cases}
\]

After an isotopy of $W$ which rescales the $t$-coordinate, we can arrange so that the movie presentation of $W$ is instead described by the planar isotopy
\[
 \Phi_{\alpha} =
  \begin{cases}
   \text{id}_P & \text{if } 0 \leq \alpha \leq 1/2 \\
   \psi_{2\alpha -1}  & \text{if } 1/2 \leq \alpha \leq 1.
  \end{cases}
\]

Now consider the composition $\Phi_{\alpha} \circ \Gamma_{\tau,\alpha}$.  Letting $\tau$ range from 0 to 1 shows that the surface $W$, which is described by the diagram $D_0$ and the planar isotopy $\Phi_{\alpha} = \Phi_{\alpha} \circ \Gamma_{0,\alpha}$, is isotopic to a surface described by $D_0$ and the planar isotopy 
\[
 \varphi_{\alpha}:=\Phi_{\alpha} \circ \Gamma_{1,\alpha}=
  \begin{cases}
   \xi_{2\alpha} & \text{if } 0 \leq \alpha \leq 1/2 \\
   \psi_{2\alpha -1} \circ \xi_{2-2\alpha} & \text{if } 1/2 \leq \alpha \leq 1.
  \end{cases}
\]

As the $\xi_{\alpha}$ is the identity outside of $\bigcup A_q$, for any $p \in S_0 \cup F_0$ and any $\alpha \in \left[0,1/2\right]$ we have $\varphi_{\alpha} (p) = \xi_{2\alpha}(p) = p$.  For $\alpha \in \left[1/2,1\right]$ and $q \in S_1 \cup F_1$ we have 
\[
\varphi_{\alpha} (q) = \psi_{2\alpha - 1} \circ \xi_{2-2\alpha} (q) = \psi_{2\alpha - 1} \circ \psi^{-1}_{1-(2-2\alpha)} (q) = q
\]
as required.  Note that all the isotopies described above fix $W_0 \cup W_1 = \partial W$.
\end{proof}

By the above lemma it is enough to prove Proposition~\ref{prop:planarisotopy} in the case when $S=S_0 = S_1$, $F=F_0 = F_1$, and all points in $S \cup F$ are fixed by $\psi_{\alpha}$.  Indeed, since the points in $S_0\cup F_0$ are stationary during the first half of the planar isotopy $\varphi_{\alpha}$, and since they form a choice of overpasses for which $D_0$ is threaded, they must also form a choice of overpasses which give rise to a threading of $D_{1/2}$ .  Likewise, $D_{1/2}$ is also threaded by $h$ with the choice of overpasses $(S_1,F_1)$, since they remain stationary for during the second half of $\varphi_{\alpha}$ and give a threading of $D_1$.  By Lemma~\ref{lem:threading} we can arrange $W$ locally near $\mathbb{R}^3 \times \{\frac{1}{2}\}$ so that $W_{1/2}$ is a closed braid with diagram $D_{1/2}$ threaded with either choice of overpasses, and prove Proposition~\ref{prop:planarisotopy} for $W_{[0,1/2]}$ and $W_{[1/2,1]}$.

Suppose then that $W$ is as above.  Although the movie presentation of $W$ does not involve any Reidemeister moves, it will (after perturbing $W$ slightly away from the boundary) contain Reidemeister II and III-like moves involving components of the diagrams and the $z$-axis $h$ (see Figure \ref{fig:reidemeisternocrossing}).  These Reidemeister-like moves are like classical Reidemeister moves, but where no crossing information is specified at double points of the projection involving $h$.  The absence of crossing information with $h$ reflects the fact that the movie presentation of $W$ does not specify the relative position of the links $W_t$ above or below $P$, and that the components of the link are free to pass through the $z$-axis during isotopies in $\mathbb{R}^3$.
\begin{figure}
 \centering
 \includegraphics[width=\textwidth]{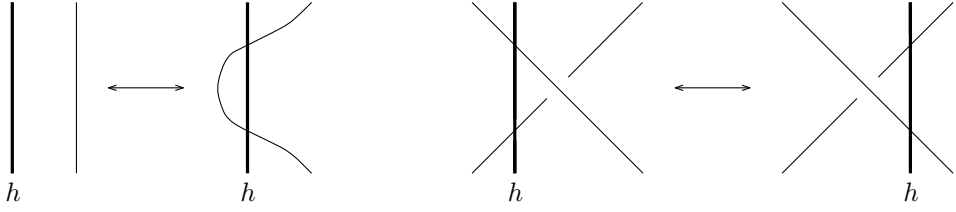}
  \caption{Reidemeister like moves involving $h$}
  \label{fig:reidemeisternocrossing}
\end{figure}

We can thus break the planar isotopy $\psi_{\alpha}$ determining $W$ into a sequence of transformations that take into account the relative position of the diagrams $D_t$ with $h$.  More precisely, we can divide the interval $\left[0,1\right]$ into smaller subintervals $\left[t_{j-1},t_j\right]$, such that for each $j$ there is either
\begin{enumerate}
\item a planar isotopy $\phi^j_{\alpha}$ of $P$, which fixes $h$ setwise and has $\phi^j_{\alpha} (D_{t_{j-1}}) = D_{t_{j-1}+\alpha (t_j-t_{j-1})}$ for all $\alpha \in \left[0,1\right]$, or
\item a Reidemeister-like move of type II or III taking $D_{t_{j-1}}$ to $D_{t_j}$ involving (but fixing) $h$.    
\end{enumerate}
We will simplify notation and write $D^j$ and $W^j$ instead of $D_{t_j}$ and $W_{t_j}$ respectively, for each $j$.  Since we are assuming that the points of $S \cup F$ are fixed throughout the planar isotopy $\psi_{\alpha}$, we can fix $(S,F)$ as a choice of overpass for each $D^j$.  Furthermore for each diagram we fix the unique choice of $h$-crossing information so that $D^j$ is threaded by $h$.

Before proceeding, we need to eliminate any situations as in Figure \ref{fig:badR3}.  Here we have a Reidemeister-like move of type III where the center crossing cannot pass to the other side of $h$ without first introducing crossing changes.  These can be eliminated by making a local replacement as in Figure \ref{fig:fixbadR3a}, where the offending move has been replaced by a sequence consisting of three Reidemeister-like moves, two of type II and one of type III (which lifts to an isotopy avoiding the $z$-axis).  This local replacement does not change the isotopy class of $W$ rel $\partial W$.

\begin{figure}
 \centering
 \includegraphics[width=0.5\textwidth]{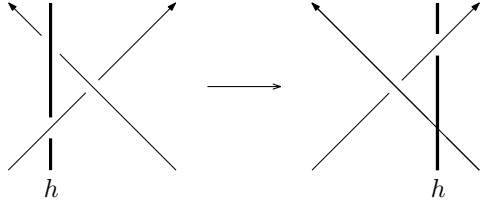}
  \caption{Reidemeister-like move of type III which does not lift to a braid isotopy}
  \label{fig:badR3}
\end{figure}

\begin{figure}
 \centering
 \includegraphics[width=\textwidth]{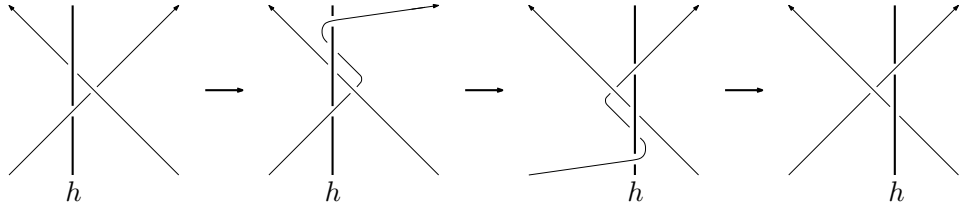}
  \caption{Replacing bad Reidemeister-like moves of type III with sequence of moves that lift to braid isotopies and simple Markov equivalences}
  \label{fig:fixbadR3a}
\end{figure}

\begin{lemma}
\label{lem:liftingdiagramtransformations}
Suppose that $W^{j-1}$ is a closed braid.  Then the transformation $D^{j-1} \rightarrow D^{j}$ lifts to $\mathbb{R}^3$ as a sequence of braid isotopies and simple Markov equivalences on $W^{j-1}$. 

\end{lemma}

\begin{proof}
Note first that since $W^{j-1}$ is a closed braid and $D^{j-1}$ is threaded, the $h$-crossing information on $D^{j-1}$ will match that coming from the projection of $W^{j-1}$. 

For transformations of type (1) above, Lemma \ref{lem:braidisotopylifting} shows that the planar isotopy between $D^{{j-1}}$ and $D^{j}$ can be lifted to a braid isotopy on $W^{{j-1}}$.

Suppose now that $D^j$ is obtained from $D^{j-1}$ by a Reidemeister-like move of type II (or its inverse) as in Figure \ref{fig:reidemeisternocrossing}.  Then as $D^{j-1}$ is threaded, locally it must look like either the right or left-hand side of one of the transformations in Figure \ref{fig:reidemeistertypeII}.  Note that by assumption no points of $S$ or $F$ can occur anywhere in these local pictures.  Clearly $D^j$ can be lifted to a closed braid $W^j$ which agrees with $W^{j-1}$ away from the Reidmeister-like move of type II, so that $W^{j-1}$ and $W^j$ are simply Markov equivalent.
\begin{figure}
 \centering
 \includegraphics[width=0.75\textwidth]{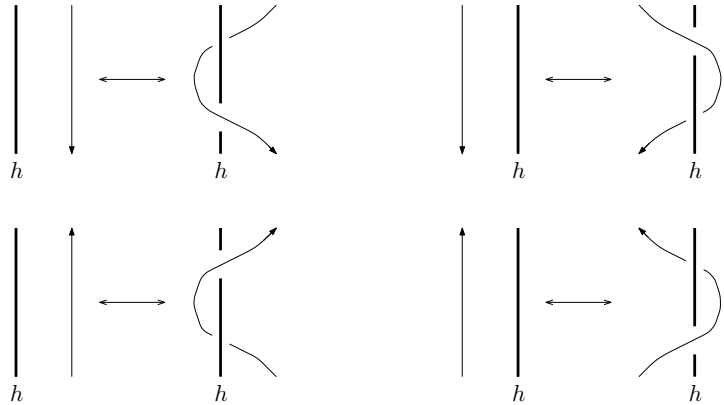}
  \caption{Reidemeister-like moves of type II}
  \label{fig:reidemeistertypeII}
\end{figure}

Now suppose that $D^j$ is obtained from $D^{j-1}$ by a Reidemeister-like move of type III.  It is easy to verify that for most configurations of $D^{j-1}$ the move can be lifted to a braid isotopy taking $W^{j-1}$ to a closed braid $W^j$ with diagram $D^j$.  The only exceptions arise as in the Figure \ref{fig:badR3}, but these were all replaced previously by sequences of moves that can be lifted.  
\end{proof}

Starting with the closed braid $W_0 \subset \mathbb{R}^3 \times \{0\}$, we can construct a new surface $W'$ by tracing the path of $W_0$ in $\mathbb{R}^3 \times \left[0,1\right]$ as we apply the sequence of lifted braid isotopies and simple Markov equivalences obtained from the previous lemma.  Away from the simple Markov equivalences each level set $W'_t$ will be a closed braid.  By construction, the movie presentation of $W'$ will be the same as that of $W$, hence it will be isotopic to $W$ rel $\partial W'$.  To prove Proposition \ref{prop:planarisotopy} it thus remains only to show that $W$ can be braided in neighborhoods of the simple Markov equivalences.          
 
\begin{proof}[Proof of Proposition \ref{prop:planarisotopy}]
Suppose that for some $s \in \left[0,1\right]$ and $\varepsilon >0$ the closed braids $W_{s-\varepsilon}$ and $W_{s+\varepsilon}$ differ by a simple Markov equivalence spanned by a disk $R$.  After a small isotopy in the neighborhood of the hyperplane $\mathbb{R}^3 \times \{s\}$ we can assume that $R$ lies entirely in this hyperplane, and that the orthogonal projection of $\partial R$ to the $xy$-plane yields a figure eight.

Decompose $R$ as the boundary sum of two closed disks $R'$ and $R''$ (equipped with the orientation of $W$), where $R'$ intersects the $z$-axis transversely in a single point and where $\partial R'$ is a simple curve which is strictly monotone in the angular direction (see Figure \ref{fig:bandanddisk}).  Push $R'$ to either $\mathbb{R}^3 \times \{s + \varepsilon\}$ or $\mathbb{R}^3 \times \{s - \varepsilon\}$ (depending on whether $\partial R'$ is monotone increasing or decreasing respectively) while keeping $R''$ fixed.  This gives rise to a new maximal disk (minimal disk respectively) while $R''$ yields a new saddle band.  After a slight local perturbation these new critical disks can be changed to isolated critical points, completing the proof of Proposition~\ref{prop:planarisotopy}.  
\end{proof}

\begin{figure}
 \centering
 \includegraphics[width=0.35\textwidth]{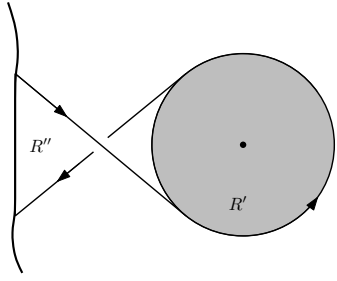}
  \caption{Decomposing $R$ as the boundary sum of $R'$ and $R''$}
  \label{fig:bandanddisk}
\end{figure}

\subsection{Braiding movie presentations with Reidemeister moves}
Now consider an arbitrary critical point free cobordism $W$ between two closed braids.  The movie presentation of $W$ under the projection to $P$ will in general include Reidemeister moves as well as planar isotopies.  Recycling notation from above, let $D_t$ denote the diagram of $W_t$, and divide the interval $\left[0,1\right]$ into smaller subintervals $\left[t_{j-1},t_j\right]$, such that for each $j$ there is either
\begin{enumerate}
\item a planar isotopy $\phi^j_{\alpha}$ of $P$ which has $\phi^j_{\alpha} (D_{t_{j-1}}) = D_{t_{j-1} + \alpha (t_j - t_{j-1})}$ for all $\alpha \in \left[0,1\right]$, or
\item a Reidemeister move taking $D_{t_{j-1}}$ to $D_{t_j}$.    
\end{enumerate}
As above we will simplify notation and write $D^j$ and $W^j$ instead of $D_{t_j}$ and $W_{t_j}$ respectively, for each $j$.  To complete the proof of Theorem \ref{thm:AlexanderMainTheorem} we need the following lemma:

\begin{lemma}
\label{lem:reidemeistermovebraid}
Suppose $D^j$ is obtained from $D^{j-1}$ by a Reidemeister move of any type.  Then there is a planar isotopy $\zeta_{\alpha}$ of $P$, such that $\zeta_1(D^{j-1})$ and $\zeta_1(D^j)$ are both threaded by $h$ for some choice of overpasses, and if $W^{j-1}$ is a closed braid with diagram $\zeta_1(D^{j-1})$, then the Reidemeister move taking $\zeta_1(D^{j-1})$ to $\zeta_1(D^j)$ lifts to a braid isotopy of $W^{j-1}$.
\end{lemma}  

To see that this completes the proof of Theorem \ref{thm:AlexanderMainTheorem}, note first that by Theorem 2 of \cite{morton} there are braid isotopies taking $W_0$ and $W_1$ to closed braids whose diagrams in $P$ are threaded by $h$ for some choices of overpasses. Thus we can assume that the diagrams $D_0$ and $D_1$ are both threaded.  We also assume that in the movie presentation of $W$ the sequence involved alternates between planar isotopies and Reidemeister moves, beginning and finishing with planar isotopies.  Suppose for some $j$ that $D^j$ is obtained from $D^{j-1}$ by a Reidemeister move, and let $\phi^{j-1}_{\alpha}$ and $\phi^{j+1}_{\alpha}$ be the planar isotopies taking $D^{j-2}$ to $D^{j-1}$ and $D^j$ to $D^{j+1}$ respectively.  Then we can replace $D^{j-1}$ and $D^j$ with $\zeta_1(D^{j-1})$ and $\zeta_1(D^j)$ respectively, and $\phi^{j-1}_{\alpha}$ and $\phi^{j+1}_{\alpha}$ with $\zeta_{\alpha} \circ \phi^{j-1}_{\alpha}$ and $\zeta_{1-\alpha} \circ\phi^{j+1}_{\alpha}$ respectively, without changing the isotopy class of $W$ rel $\partial W$.  Performing a similar replacement one by one around all Reidemeister moves in the movie presentation, we see that $W$ is isotopic relative its boundary to a cobordism whose movie presentation involves only Reidemeister moves and planar isotopies between threaded diagrams.  

Thus we can assume that each of the $D^j$ are threaded and that the $W^j$ are all closed braids.  By Lemma \ref{lem:reidemeistermovebraid} the portions of $W$ corresponding to planar isotopies in the movie presentation are then isotopic relative their boundaries to braided cobordisms, while by Proposition \ref{prop:planarisotopy} we see that the same is true for portions of $W$ corresponding to Reidemeister moves.  Thus $W$ itself is isotopic relative its boundary to a braided cobordism, completing the proof.  

\begin{proof}[Proof of Lemma \ref{lem:reidemeistermovebraid}]
Begin by making a choice of overpasses for $D^{j-1}$ and $D^j$ which agree outside some small neighborhood of the move in question.  In the small neighborhood of the move we choose points which give a valid choice of overpasses both before and after the move.  See examples of different possible configurations in Figure \ref{fig:Rmovesoverpasses}, where incoming strands are labeled with $o$ if they are part of an overpass, or $u$ if they are part of an underpass.
\begin{figure}
 \centering
 \includegraphics[width=0.6\textwidth]{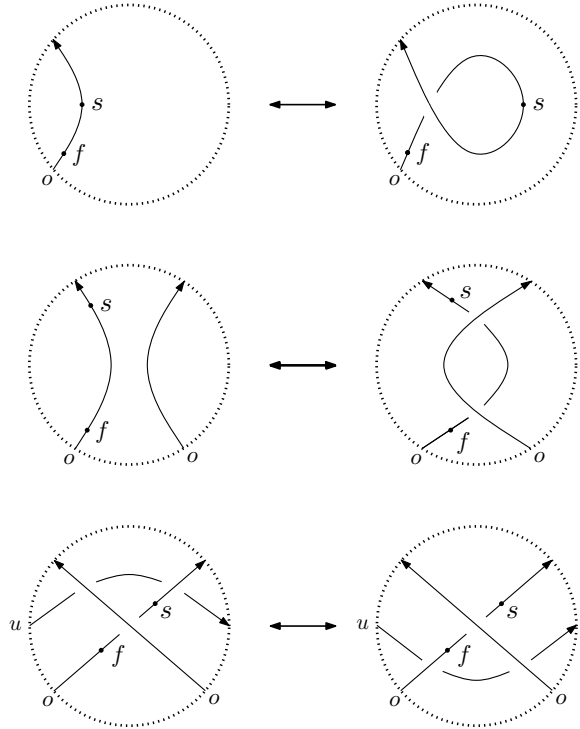}
  \caption{Overpass choices in a neighborhood of type I and II moves}
  \label{fig:Rmovesoverpasses}
\end{figure}  

Now let $\zeta_{\alpha}$ be a planar isotopy which repositions all of the $S$ points to $P_-$ (the left half of the plane $P$), and all the $F$ points to $P_+$ (the right half of $P$).  Once positioned in this way, there is a unique way to assign over and undercrossings of $D^{j-1}$ and $D^j$ with $h$ so that both diagrams are threaded by $h$.     

Note that in the case of moves of type I and II, we can choose $S,F$, and $\zeta_{\alpha}$ so that the Reidemeister move of interest happens away from $h$.  It is then easy to see that the Reidemeister move of interest lifts to a braid isotopy.

Moves of type III cannot be arranged to take place away from $h$ however.  Of the three strands in this local picture, one strand will cross over the other two, one will pass under the other two, while the third will pass over one and under the other.  Choose $S$ and $F$ away from this picture so that the top strand is part of an overcrossing, the bottom strand is part of an undercrossing, and place a single point from each of $S$ and $F$ on the third strand to create a valid choice of overpasses. 
\begin{figure}
 \centering
 \includegraphics[width=0.56\textwidth]{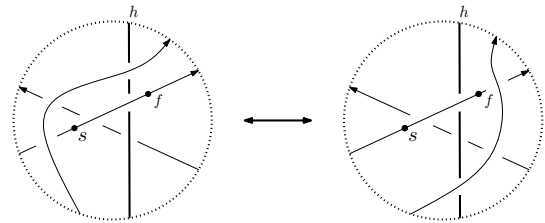}
  \caption{Threading near a Reidemeister move of type III}
  \label{fig:R3threading}
\end{figure} 

Now we can arrange the diagrams so that $h$ separates $S$ and $F$, and so that the uppermost strand crosses over $h$ in a neighborhood of the move (the orientation of this strand determines whether it will cross $h$ at the top or bottom of the local picture).  Regardless then of the orientation on the other two strands or their shared crossing, the uppermost strand is free to pass over the crossing and both the nearby $S$ and $F$ points as in Figure \ref{fig:R3threading}, a move which can clearly be lifted to a braid isotopy in $\mathbb{R}^3$.  This completes the proof of Lemma \ref{lem:reidemeistermovebraid} and of Theorem \ref{thm:AlexanderMainTheorem}.
\end{proof}

Corollary~\ref{cor:braidedsurfaceswithcaps} now follows easily by combining Theorem~\ref{thm:AlexanderMainTheorem} with Lemma~\ref{lem:braidedcobordismtobraidedsurfacewithcaps}.

\begin{remark} 
Suppose now that the cobordism $W$ we start with is in ribbon position, i.e., has no local maximal points with respect to the $t$-coordinate.  Although we may hope to preserve this property during the braiding procedure described above, this will not be possible in general.  Indeed, Morton \cite{Morton1983} gave an example of a 4-strand braid $\beta$ with unknotted closure which is \emph{irreducible}, meaning any simplification of $\beta$ using Markov moves necessarily raises the braid index to 5.  As noted by Rudolph \cite{Rudolph1985}, it is not difficult to see that any braided ribbon cobordism bounded by the closure of $\beta$ must have genus $\geq 1$, even though it clearly bounds a ribbon embedded disk in $S^3 \times [0,1]$. 
\end{remark}

\bibliographystyle{plain}
\bibliography{bibliography}
\end{document}